\documentclass[12pt,a4paper]{amsart}

\usepackage{amsmath}
\usepackage{amsthm,amssymb,amsfonts}

\usepackage{graphicx}

\newtheorem{thm}{Theorem}[section]
\newtheorem{lemma}[thm]{Lemma}
\newtheorem{prop}[thm]{Proposition}
\newtheorem{cor}[thm]{Corollary}

\theoremstyle{definition}

\newtheorem{remark}[thm]{Remark}

\def\eqn#1$$#2$${\begin{equation}\label#1#2\end{equation}}

\def\th{{\leavevmode\setbox1=\hbox{t}
  \hbox to \wd1{t\kern-0.6ex{\char039}\hss}}}
\def\dh{{\leavevmode\setbox1=\hbox{d}
  \hbox to 1.05\wd1{d\kern-0.4ex{\char039}\hss}}}
\def\=#1{\if #1u{\accent23u}\else
\ifx #1d{\dh}\else \ifx #1t{\th}\else
 {\accent20 #1}\fi\fi\fi}
\def\'#1{\if #1i{\accent19\i}\else {\accent19 #1}\fi}

\def\en{\mathbb N}
\def\er{\mathbb R}

\def \reg {\partial _{\kern1pt\text{reg}}}

\def\wscl#1{\overline{#1}^{w^*}}

\newtoks\by
\newtoks\paper
\newtoks\book
\newtoks\jour
\newtoks\yr
\newtoks\pages
\newtoks\vol
\newtoks\publ
\newtoks\eds
\newtoks\proc
\newtoks\mathrev
\newtoks\web
\def\ota{{\hbox{???}}}
\def\cLear{\by=\ota\paper=\ota\book=\ota\jour=\ota\yr=\ota
\pages=\ota\vol=\ota\publ=\ota}
\def\endpaper{\the\by, \textit{\the\paper},
{\the\jour} \textbf{\the\vol} (\the\yr), \the\pages.\cLear}
\def\endbook{\the\by, \textit{\the\book}, \the\publ, \the\yr.\cLear}
\def\endprep{\the\by, \textit{\the\paper}, \the\jour.}
\def\endprepkma{\the\by, \textit{\the\paper}, \the\jour, available on http://adela.\-karlin.\-mff.\-cuni.\-cz/\~{ }rokyta/\-preprint/\-index.php.\cLear}
\def\endproc{\the\by, \textit{\the\paper}, \the\book,
\the\publ, \the\yr, \the\pages.\cLear}
\def\endper{\the\by, \textit{personal communication}.\cLear}

\title[Weak aproximate fixed point property]
{Spaces not containing $\ell_1$ have weak aproximate fixed point property}

\author[O.F.K. Kalenda]{Ond\v{r}ej  F.K. Kalenda }
    \address{Department of Mathematical Analysis \\
Faculty of Mathematics and Physic\\ Charles University\\
Sokolovsk\'{a} 83, 186 \ 75\\Praha 8, Czech Republic}
 \email{kalenda@karlin.mff.cuni.cz}

\thanks{The research was supported in part by the grant
GAAV IAA 100190901  and in part by the Research Project
MSM~0021620839 from the Czech Ministry of Education.}

\date{}


\begin{document}

\begin{abstract}
A nonempty closed convex bounded subset $C$ of a Banach space is said to have the weak approximate fixed point property if for every continuous map $f:C\to C$ there is a sequence $\{x_n\}$ in $C$ such that $x_n-f(x_n)$ converge weakly to $0$. We prove in particular that $C$ has this property whenever it contains no sequence equivalent to the standard basis of $\ell_1$. As a byproduct we obtain a characterization of Banach spaces not containing $\ell_1$ in terms of the weak topology.
\end{abstract}

\maketitle

\section{Introduction and main results}

Let $X$ be a real Banach space and $C$ a nonempty closed convex bounded subset of $X$. The set $C$ is said to have the {\it approximate fixed point property} (shortly {\it afp property}) if for every continuous mapping 
$f:C\to C$ there is a sequence $\{x_n\}$ in $C$ such that $x_n-f(x_n)\to0$. The set $C$ is said to have the {\it weak approximate fixed point property} (shortly {\it weak afp property}) if for every continuous mapping 
$f:C\to C$ there is a sequence $\{x_n\}$ in $C$ such that the sequence $\{x_n-f(x_n)\}$ weakly converges to $0$. 

The study of
these notions was started by C. S. Barroso \cite{Ba} in topological vector spaces
where, in particular, the weak afpp for weakly compact convex subsets of
Banach spaces was proved, and after by C.S. Barroso and P.-K. Lin \cite{BL} in
Banach spaces for general bounded, closed convex sets with emphazis on
geometrical aspects.

Our terminology follows \cite{BL}. Anyway, it is worth to remark that
the notion of the afp property in this context does not have a good meaning.
Indeed, if $C$ is compact, then any continuous selfmap of $C$ has even a fixed point by Schauder's theorem (see e.g. \cite[p. 151, Theorem 183]{HHZ}). If $C$ is not compact, then it does not have the afp property by a result of P.-K.\ Lin and Y.\ Sternfeld \cite[Theorem 1]{LS}.  
Anyway it may have a sense in case of non-complete $X$ or non-closed $C$. A Lipschitz version of this property is studied in \cite{LS}. 

For the weak afp property the situation is different:

A Banach space $X$ is said to have the {\it weak approximate fixed point property} if each nonempty closed convex bounded subset of $X$ has the weak afp property.

This notion was studied by C.S.\ Barroso and P.-K.\ Lin in \cite{BL}. They proved that Asplund space do have the weak afp and asked in Problem 1.1 whether the same is true for spaces not containing $\ell_1$. In the present paper we answer this question affirmatively. This is the content of the following theorem.

\begin{thm}
Let $X$ be a Banach space. Then $X$ has the weak approximate fixed point property if and only if $X$ contains no isomorphic copy of $\ell_1$.
\end{thm}

This theorem is an immediate consequence of the following more general theorem:

\begin{thm}\label{convex} 
Let $X$ be a Banach space and $C$ a nonempty closed convex bounded subset of $X$. Then the following assertions are equivalent.
\begin{itemize}
	\item[(1)] Each nonempty closed convex subset of $C$ has the weak approximate fixed point property.
	\item[(2)] $C$ contains no sequence equivalent to the standard basis of $\ell_1$.
\end{itemize}
\end{thm}

Let us recall that a bounded sequence $\{x_n\}$ is equivalent to the standard basis of $\ell_1$ if there is a constant $c>0$ such that for any $N\in\en$ and any choice of $a_1,\dots,a_N\in\er$ we have
$$\left\|\sum_{n=1}^N a_n x_n\right\|\ge c\sum_{n=1}^N |a_n|.$$
It means that the mapping $T:\ell_1\to X$ defined by $T(\{a_n\})=\sum_{n=1}^\infty a_n x_n$ is an isomorphic embedding. Such sequences $\{x_n\}$ will be called {\it $\ell_1$-sequences}.

The implication $(1)\Rightarrow(2)$ is known to be true. Indeed, suppose that $\{x_n\}$ is an $\ell_1$-sequence contained in $C$. Set $D$ to be the closed convex hull of the set $\{x_n:n\in\en\}$.
Let $T$ be the mapping defined in the previous paragraph. Then $Y=T(\ell_1)$ is a subspace of $X$ which is isomorphic to $\ell_1$ and contains $D$. So, by Schur's theorem (see e.g. \cite[p. 74, Theorem 99]{HHZ}), weakly convergent sequences in $Y$ are norm convergent. So, if $D$ had the weak afp property, it would have the afp property as well.
But it is impossible by the already quoted \cite[Theorem 1]{LS} as $D$ is not compact.

We remark that Theorem~\ref{convex} immediately implies that weakly compact sets have weak afp property which also follows from a result of C.S.\ Barroso \cite[Theorem 3.1]{Ba}.

We finish this section by recalling and commenting two results from \cite{BL}.

\begin{lemma}\label{L1} Let $X$ be any Banach space, $C$ any nonempty closed convex bounded subset of $X$ and $f:C\to C$ any continuous mapping. Then the point $0$ is in the weak closure of the set $\{x-f(x): x\in C\}$.
\end{lemma}

This lemma is proved in \cite[Lemma 2.1]{BL} using Brouwer's fixed point theorem and paracompactness of metric spaces. 

\begin{lemma}\label{L2} Let $X$ be any Banach space, $C$ any nonempty closed convex bounded subset of $X$ and $f:C\to C$ any continuous mapping. Then there is a nonempty closed convex separable subset $D\subset C$ with $f(D)\subset D$.
\end{lemma}

This is easy and is proved in the second part of the proof of Theorem 2.2 in \cite{BL}.

In view of Lemma~\ref{L1} to prove the weak afp property one needs to reach the point $0$ by a limit of a sequence, not just by a limit of a net. In \cite[Theorem 2.2]{BL} it is done by using implicitly the metrizability of the weak topology on bounded sets of a separable Asplund space. We show that it is also possible under the weaker assumption that the space does not contain a copy of $\ell_1$. Topological results which enable us to do so are contained in the following section.

\section{$\ell_1$-sequences and Fr\'echet-Urysohn property of the weak topology}

Let us recall that a topological space $T$ is called {\it Fr\'echet-Urysohn} if the closures of subsets of $T$ are described using sequences, i.e. if whenever $A\subset T$ and $x\in T$ is such that $x\in\overline{A}$, there is a sequence $\{x_n\}$ in $A$ with $x_n\to x$. Metrizable spaces are Fr\'echet-Urysohn but there are many nonmetrizable Fr\'echet-Urysohn spaces (for examples see the results below).

We will need the following deep result of J.\ Bourgain,  D.H.\ Fremlin and M.\ Talagrand \cite[Theorem 3F]{BFT}:

\begin{thm}\label{t-BFT} Let $P$ be a Polish space (i.e., a separable completely metrizable space). Denote by $B_1(P)$ the space of all real-valued functions on $P$ which are of the first Baire class and equip this space with the topology of pointwise convergence. Suppose that $A\subset B_1(P)$ is relatively countably compact in $B_1(P)$ (i.e., each sequence in $A$ has a cluster point in $B_1(P)$). Then the closure $\overline A$ of $A$ in $B_1(P)$ is compact and Fr\'echet-Urysohn.
\end{thm}

In fact, we need a slightly weaker version formulated in the following corollary.

\begin{cor}\label{cor-BFT} Let $P$ be a Polish space and $A$ be a set of real-valued continuous functions on $P$. Suppose that each sequence in $A$ has a pointwise convergent subsequence. Then the closure of $A$ in $\er^P$ is a Fr\'echet-Urysohn compact space contained in $B_1(P)$.
\end{cor}

\begin{proof} $A$ is obviously contained in $B_1(P)$. Moreover, let $(f_n)$ be any sequence in $A$. By the assumption there is a subsequence $(f_{n_k})$ pointwise converging to some function $f$. As the functions $f_{n_k}$ are continuous, the limit function $f$ is of the first Baire class. Hence, it is a cluster point of $(f_n)$ in $B_1(P)$. So, $A$ is relatively countably compact in $B_1(P)$. The assertion now follows from Theorem~\ref{t-BFT}.
\end{proof}

Now we are ready to prove the following proposition which can be viewed as an improvent of a result due to E.\ Odell and H.P.\ Rosenthal \cite{OR} on characterization of separable spaces not containing $\ell_1$. We note that we use the results of \cite{BFT} and this paper was published three years after \cite{OR}.

\begin{prop}\label{FU} Let $X$ be a Banach space and $C$ be a bounded subset of $X$. 
If $C$ is norm-separable and contains no $\ell_1$-sequence, then the set
$$\wscl{\kappa(C-C)}=\wscl{\{\kappa(x-y): x,y\in C\}}$$
is Fr\'echet-Urysohn when equipped with the weak* topology, where $\kappa$ denotes the canonical embedding of $X$ into $X^{**}$. In particular,
$$\overline{C-C}^{w}=\overline{\{x-y: x,y\in C\}}^{w}$$
is Fr\'echet-Urysohn when equipped with the weak topology.
\end{prop}

\begin{proof} As the closed linear span of $C$ is separable, we can without loss of generality suppose that $X$ is separable. Further we have:
$$\mbox{Each sequence in $C-C$ has a weakly Cauchy subsequence.}\eqno{(*)}$$
Indeed, let $\{z_n\}$ be a sequence in $C-C$. Then there are sequences $\{x_n\}$ and $\{y_n\}$ in $C$ such that $z_n=x_n-y_n$ for each $n\in\en$. As $C$ contains no $\ell_1$-sequence, by Rosenthal's theorem \cite{R} there is a subsequence $\{x_{n_k}\}$ of $\{x_n\}$ which is weakly Cauchy. Using Rosenthal's theorem once more, we get a subsequence $\{y_{n_{k_l}}\}$ of $y_{n_k}$ which is weakly Cauchy. Then $\{z_{n_{k_l}}\}=\{x_{n_{k_l}}-y_{n_{k_l}}\}$ is a weakly Cauchy subsequence of $\{z_n\}$. This completes the proof of \thetag{*}.

Further, denote by $K$ the dual unit ball $(B_{X^*},w^*)$ equipped with the weak* topology. Then $K$ is a metrizable compact space. Denote by $r$
the mapping $r:X^{**}\to \er^K$ defined by $r(F)=F|_K$ for $F\in X^{**}$. Then we have:
\begin{itemize}
	\item[(i)] $r$ is a homeomorphism of $(X^{**},w^*)$ onto $r(X^{**})$.
	\item[(ii)] $r\circ \kappa$ is a homeomorphism of $(X,w)$ onto $r(\kappa(X))$.
	\item[(iii)] The functions from $r(\kappa(X))$ are continuous on $K$. 
\end{itemize}
Set $M=r(\kappa(C-C))$. Then $M$ is a uniformly bounded sets of continuous functions on $K$.
Moreover, by \thetag{*} any sequence from $M$ has a pointwise convergent subsequence. 
By Corollary~\ref{cor-BFT} the closure of $M$  in $\er^K$ is a Fr\'echet-Urysohn compact subset of $B_1(K)$. But this closure is equal to $r\left(\wscl{\kappa(C-C)}\right)$. It follows that $\wscl{\kappa(C-C)}$ is Fr\'echet-Urysohn in the weak* topology. This completes the proof of the first statement.

Further, to show the `in particular case' it is enough to observe that the set $\wscl{\kappa(C-C)}$ contains $\kappa\left(\overline{C-C}^{w}\right)$, hence $\overline{C-C}^{w}$ is Fr\'echet-Urysohn in the weak topology.
\end{proof} 

As a corollary we get the following characterization of spaces not containing $\ell_1$:

\begin{thm}\label{ell1} Let $X$ be a Banach space. Then the following assertions are equivalent.
\begin{itemize}
	\item[(1)] $X$ contains no isomorphic copy of $\ell_1$.
	\item[(2)] Each bounded separable subset of $X$ is Fr\'echet-Urysohn in the weak topology.
	\item[(3)] For each separable subset $A\subset X$ there are relatively weakly closed subsets $A_n$, $n\in\en$, of $A$ such that $A=\bigcup_{n\in\en}A_n$ and each $A_n$ is Fr\'echet-Urysohn in the weak topology.  
\end{itemize}
\end{thm}

Note that the assertion (3) is a topological property of the space $(X,w)$ (as norm separability coincides with weak separability).

\begin{proof}
The implication (1)$\Rightarrow$(2) follows from Proposition~\ref{FU}. 

The implication (2)$\Rightarrow$(1) follows from the fact that the unit ball of $\ell_1$ is not Fr\'echet-Urysohn (as $0$ is in the weak closure of the sphere and the sphere is weakly sequentially closed by the Schur theorem).

The implication (2)$\Rightarrow$(3) is trivial if we use the fact that a closed ball is weakly closed.

Let us prove (3)$\Rightarrow$(2). To show (2) it is enough to prove that the unit ball of any closed separable subspace of $X$ is Fr\'echet-Urysohn in the weak topology. Let $Y$ be such a subspace. Let $Y_n$, $n\in\en$, be the cover of $Y$ provided by (3). As each $Y_n$ is weakly closed, it is also norm-closed. By Baire category theorem some $Y_n$ has a nonempty interior in $Y$, so it contains a ball.
We get that some ball in $Y$ is Fr\'echet-Urysohn, so the unit ball has this property as well.
\end{proof}

It is worth to compare the previous theorem with a similar characterization of Asplund spaces:

\begin{thm} Let $X$ be a Banach space. Then the following assertions are equivalent.
\begin{itemize}
	\item[(1)] $X$ is Asplund.
	\item[(2)] Each bounded separable subset of $X$ is metrizable in the weak topology.
	\item[(3)] For each separable subset $A\subset X$ there are relatively weakly closed subsets $A_n$, $n\in\en$, of $A$ such that $A=\bigcup_{n\in\en}A_n$ and each $A_n$ is metrizable in the weak topology. \end{itemize}
\end{thm}

We recall that $X$ is Asplund if and only if $Y^*$ is separable for each separable subspace $Y\subset X$. The equivalence of (1) and (2) follows from the well-known fact that the unit ball of $Y$ is metrizable in the weak topology if and only if $Y^*$ is separable. The equivalence of (2) and (3) can be proved similarly as corresponding equivalence in the previous theorem.

\begin{remark} There is no analogue of Theorem~\ref{ell1} for convex sets.
Indeed, let $X=\ell_1$ and $C$ be the closed convex hull of the standard basis. Then $C$ contains an $\ell_1$-sequence but is Fr\'echet Urysohn in the weak topology. In fact, it is even metrizable as it is easy to see that on the positive cone of $\ell_1$ the weak and norm topologies coincide.
\end{remark}

\section{Proof of Theorem~\ref{convex}}

It remains to prove the implication (2)$\Rightarrow$(1). 
Let $X$ be a Banach space, $C\subset X$ a nonempty closed convex bounded set containing no $\ell_1$-sequence and $f:C\to C$ be a continuous mapping. Let $D$ be the set provided by Lemma~\ref{L2}.
Then $D$ is separable and contains no $\ell_1$-sequence. By Lemma~\ref{L1} the point $0$ is in the weak closure of $\{x-f(x):x\in D\}$. By Proposition~\ref{FU} there is a sequence from this set weakly converging to $0$. This completes the proof.\qed

\bigskip

\begin{remark} We stress the difference between approximation in the norm and in the weak topology. Suppose that $X$ is a Banach space, $C\subset X$ a nonempty closed convex bounded set and $f:C\to C$ a continuous mapping.

For approximation in the norm, we have the equivalence of the following three conditions:
\begin{itemize}
	\item There is a sequence $\{x_n\}$ in $C$ such that $x_n-f(x_n)\to0$.
	\item The point $0$ is in the norm-closure of the set $\{x-f(x):x\in C\}$.
	\item $\inf\{\|x-f(x)\|: x\in C\}=0$.
\end{itemize}

These three statements are trivially equivalent (by properties of metric spaces) and are rather strong.
For the weak topology the situation the situation is different. First, there is no analogue of the third condition. Secondly, the analogue of the second one is satisfied allways by Lemma~\ref{L1}.
But the analogue of the first one is not satisfied allways, as the weak topology is not in general
described by sequences. 
\end{remark}

\section*{Acknowledgement}
The author is grateful to Barry Turett for informing him about the problem. Thanks are also due
to Mari\'an Fabian for providing the contact to Barry Turett.

 \end{document}